\documentclass[11pt]{amsart}
\usepackage{amsthm}
\usepackage{amscd}  
\usepackage{graphicx}
\usepackage{amssymb}
\usepackage{epstopdf}

\theoremstyle{plain}
\newtheorem{Thm}{Theorem}[section]

\newtheorem{Prop}[Thm]{Proposition}
\newtheorem{Cor}[Thm]{Corollary}
\newtheorem{Lem}[Thm]{Lemma}

\theoremstyle{definition}
\newtheorem{Defn}[Thm]{Definition}

\numberwithin{equation}{section}

\title{Non-commutative Grassmann variety as a moduli space}
\author{Yujiro Kawamata}
\date{}                                           
\begin{document}
\maketitle

\begin{abstract}
We construct a non-commutative version of the Grassmann variety $G(2,4)$
as a non-commutative moduli space of linear subspaces in a projective space.

14A22, 14M15.
\end{abstract}

\section{Introduction}

We will construct a non-commutative (NC) version of the Grassmann variety $G(m,n)$, denoted by $NCG(m,n)$, 
as an NC moduli space of linear subspaces in a projective space in the case $m = 2$ and $n = 4$.
$NCG(m,n)$ is expected to have the following properties:

(0) $NCG(m,n)$ is an NC scheme.
An NC scheme is defined by gluing associative algebras.
A usual scheme has a structure sheaf, but an NC scheme is a presheaf, i.e., 
a functor from a poset to the category of associative algebras, 
because there are no localizations of associative algebras in general.

(1) The set of closed points of $NCG(m,n)$ is the same as $G(m,n)$.
The set of closed points of an NC scheme is usually small, because a closed point is defined as a homomorphism to
a commutative algebra and the information on the non-commutativity is lost.  

(2) The completion of $NCG(m,n)$ at each closed point parametrizes a semi-universal formal NC deformation 
of the structure sheaf $\mathcal O_L$ for the corresponding linear subspace $L \subset \mathbf P^{n-1}$.

(3) There is a universal NC deformation family of linear subspaces of $\mathbf P^{n-1}$ parametrized by $NCG(m,n)$.

We note that $NCG(m,n)$ is {\em not} an NC deformation of $G(m,n)$ as a variety.
For example, it may not be locally Noetherian.
This is because of the fact that there are much more NC deformations than the usual deformations of the linear subspaces.

$NCG(m,n)$ is a kind of global moduli space of a sheaf on a fixed projective space $\mathbf P^{n-1}$.
There are different kinds of variables of the associative algebras which are either commutative, 
semi-commutative or totally non-commutative.
For example, $NCG(1,n) = \mathbf P^{n-1}$ has only commutative variables, 
while $NCG(n - 1,n)$ is totally NC and not noetherian.
If $1 < m < n-1$, then the variables are mixed.
There are commutative pairs of variables as well as commutative quartets.
The point of the construction of the NC scheme is that the commutator relations are compatible with the change of variables
between local charts.

The following is the main result:

\begin{Thm}
Assume that $m = 2$ and $n = 4$.
Then there exists an NC scheme $NCG(2,4)$ which has the above properties (1), (2) and (3).
\end{Thm}

We recall the definition of NC schemes in \S 2 and the construction of a local model for $NCG(m,n)$ in \S 3.
The theorem is proved in \S 4.

We work over an algebraically closed base field $k$ throughout the paper.

The author would like to thank NCTS of National Taiwan University where 
this work was partly done while the author visited there.
This work is partly supported by JSPS Kakenhi 21H00970.

\section{NC scheme}

We recall a definition of an NC scheme (\cite{DLL}, \cite{smooth}).
It is a presheaf of associative algebras in the following sense.
Let $\Lambda$ be a poset which has a minimum for arbitrary pair of elements.
$\Lambda$ is considered to be a category such that $\text{Hom}(\lambda_1, \lambda_2)$ is a set with one element 
if $\lambda_1 \ge \lambda_2$ and empty otherwise.
An {\em NC scheme} is a functor $F: \Lambda \to \text{Alg}$ from $\Lambda$ to the category of unital associative algebras.
$\lambda \in \Lambda$ is considered to be the index of an imaginary affine open subset $U_{\lambda}$,  
and $F(\lambda)$ the corresponding coordinate ring.
$\lambda_1 > \lambda_2$ corresponds to the inclusion $U_{\lambda_1} \supset U_{\lambda_2}$, 
and the ring homomorphism $\phi_{\lambda_2,\lambda_1}: F(\lambda_1) \to F(\lambda_2)$ corresponds to 
the restriction of functions.
If $\lambda_1 > \lambda_2 > \lambda_3$, then we have $\phi_{\lambda_3,\lambda_2}\phi_{\lambda_2,\lambda_1} 
= \phi_{\lambda_3,\lambda_1}$ by the axiom of functors.

\section{local model of NC Grassmann}

Let $\mathbf P^{n-1}$ be a projective space with homogeneous coordinates $x_1,\dots,x_n$.
The usual Grassmann variety $G(m,n)$ parametrizes $(m-1)$-dimensional linear subspaces in $\mathbf P^{n-1}$.
It is covered by affine open subsets $U_{\lambda}$ where $\lambda \subset \{1,\dots,n\}$ is a
subset of order $m$, which parametrize
linear subspaces whose equations are of the form $x_j = \sum_{i \in \lambda} a_{i, j}x_i$ for $j \not\in \lambda$ and
$a_{i, j} \in k$.
$U_{\lambda} \cong k^{m(n-m)}$ with coordinates $a_{i,j}$ for $i \in \lambda$ and $j \not\in \lambda$.

The index set $\Lambda$ for the NC Grassmann variety $NCG(m,n)$ is defined as follows.
A maximal element $\lambda = \{i_1,\dots, i_m\} \in \Lambda$ is a subset of order $m$ of the set $\{1,\dots, n\}$.
There are $\binom nm$ maximal elements.
The index set $\Lambda$ consists of maximal elements and their minima 
$\min \{\lambda_1, \dots, \lambda_r\}$ for maximal $\lambda_1, \dots, \lambda_r$ which correspond 
to the intersections $\bigcap_{i=1}^r U_{\lambda_i}$ of affine open subsets on $G(m,n)$.
We note that $\min \{\lambda_1, \dots, \lambda_r\}$ is {\em not} the intersection of subsets of $\{1,\dots, n\}$
denoted by $\bigcap_{i=1}^r \lambda_i$.

The local model $R_{\lambda}$ of $NCG(m,n)$ corresponding to $U_{\lambda}$ is already described in
\cite{NCbase}.
$R_{\lambda}$ is an associative algebra whose abelianization 
$R_{\lambda}^{ab} = R_{\lambda}/[R_{\lambda},R_{\lambda}]$ is isomorphic to the coordinate ring of 
$U_{\lambda}$.

\begin{Defn}
\[
R_{\lambda} = k\langle a_{i,j}^{\lambda} \mid i \in \lambda, j \not\in \lambda \rangle/I_{\lambda}
\]
where the $a_{i,j}$ are independent variables and 
$I_{\lambda}$ is a two-sided ideal generated by the following commutativity relations:
\[
\begin{split}
&a_{i,j_1}^{\lambda}a_{i,j_2}^{\lambda} - a_{i,j_2}^{\lambda}a_{i,j_1}^{\lambda}, \\
&a_{i_1,j_1}^{\lambda}a_{i_2,j_2}^{\lambda} - a_{i_2,j_2}^{\lambda}a_{i_1,j_1}^{\lambda} 
- a_{i_1,j_2}^{\lambda}a_{i_2,j_1}^{\lambda} + a_{i_2,j_1}^{\lambda}a_{i_1,j_2}^{\lambda}
\end{split}
\]
for all $i, i_1,i_2 \in \lambda$ and $j,j_1,j_2 \not\in \lambda$.
We define a left $R_{\lambda}$-module $F_{\lambda}$, called a {\em universal module}, 
as a quotient of $R_{\lambda} \otimes k[x_1,\dots, x_n]$ 
by a two-sided ideal $J_{\lambda}$ generated by the following equations:
\[
x_j - \sum_{i \in \lambda} a_{i,j}^{\lambda}x_i
\]
for all $j \not\in \lambda$.
\end{Defn}

The relations say that there are commutative pairs $(a_{i,j_1},a_{i,j_2})$ and commutative quartets 
$(a_{i_1,j_1},a_{i_2,j_2},a_{i_1,j_2},a_{i_2,j_1})$.
The set of closed points, i.e., the maximal two-sided ideals, of $R_{\lambda}$ is the same as that of 
$U_{\lambda}$.
Indeed the abelianization $R_{\lambda}^{\text{ab}} = k[\bar a_{i,j}^{\lambda} \mid i \in \lambda, j \not\in \lambda]$
is a polynomial ring, where the $\bar a_{i,j}^{\lambda}$ are the images of the $a_{i,j}^{\lambda}$, 
and the set of closed points is isomorphic to an affine space $\mathbf A^{m(n-m)}$.
The completion of $R_{\lambda}$ at any closed point $p$ given by the two-sided ideal
$(a_{i,j}^{\lambda} - a_{i,j}^0)_{i,j}$ ($a_{i,j}^0 \in k$) is the parameter algebra 
of the semi-universal NC deformation of the sheaf $\mathcal O_{L_p}$ for a linear subspace 
$L_p \subset \mathbf P^{n-1}$ corresponding to the point $p$ (\cite{NCbase}).

\section{$G(2,4)$}

We will treat $NCG(2,4)$ and prove our theorem.
For $\lambda \in \{1,2,3,4\}$ with $\vert \lambda \vert = 2$, we define
$R_{\lambda} = k\langle a^{\lambda}_{i,j} \vert i \not\in \lambda, j \in \lambda \rangle/I_{\lambda}$
as before.
We will define the algebras corresponding to the intersections of open subsets of $G(2,4)$
and investigate the change of variables formulas.

The set of maximal elements $\Lambda^{\max}$ has $6$ elements ($\vert \Lambda^{\max} \vert = 6$), 
and these elements can be put at the vertexes of a regular octahedron, 
where two vertexes are joined by an edge if and only if $\vert \lambda \cap \lambda' \vert = 1$.
For any $\lambda \in \Lambda^{\max}$, there are $4$ elements 
$\lambda' \in \Lambda^{\max}$ such that
$\vert \lambda \cap \lambda' \vert = 1$ and only $1$ such that $= 0$.

By the symmetry, it will be sufficient to consider the algebras and change of variables among 
$\lambda_1 = \{1,2\}$, $\lambda_2 = \{2,3\}$ and $\lambda_3 = \{3,4\}$.
We write $R_{k} = R_{\lambda_k}$. 
We write also $a_{i,j} = a_{i,j}^{\lambda_1}$, $a'_{i,j} = a_{i,j}^{\lambda_2}$, $a''_{i,j} = a_{i,j}^{\lambda_3}$, 
$I = I_{\lambda_1}$, $I' = I_{\lambda_2}$ and $I'' = I_{\lambda_3}$.
Thus $I$, $I'$ and $I''$ are generated by the following relations:
\[
\begin{split}
&a_{1,3}a_{1,4} - a_{1,4}a_{1,3} = 0, \,\,\, a_{2,3}a_{2,4} - a_{2,4}a_{2,3} = 0, \\
&a_{1,3}a_{2,4} - a_{2,4}a_{1,3} + a_{2,3}a_{1,4} - a_{1,4}a_{2,3} = 0, \\
&a'_{2,4}a'_{2,1} - a'_{2,1}a'_{2,4} = 0, \,\,\, a'_{3,4}a'_{3,1} - a'_{3,1}a'_{3,4} = 0, \\
&a'_{2,4}a'_{3,1} - a'_{3,1}a'_{2,4} + a'_{3,4}a'_{2,1} - a'_{2,1}a'_{3,4} = 0, \\
&a''_{3,1}a''_{3,2} - a''_{3,2}a''_{3,1} = 0, \,\,\, a''_{4,1}a''_{4,2} - a''_{4,2}a''_{4,1} = 0, \\
&a''_{3,1}a''_{4,2} - a''_{4,2}a''_{3,1} + a''_{4,1}a''_{3,2} - a''_{3,2}a''_{4,1} = 0.
\end{split}
\]

\begin{Defn}\label{R12}
Define
\[
R_{1,2} = k\langle a_{i,j}, a'_{i',j'}, a_{1,3}^{-1}, (a'_{3,1})^{-1} \vert i = 1,2,\, j = 3,4, \, i' = 2,3, \, j' = 4,1 \rangle/I_{1,2}
\]
where $I_{1,2}$ is a two-sided ideal generated by the following relations:
\[
\begin{split}
&(1) \,\,\, I, \,\, I', \,\, a_{1,3}a_{1,3}^{-1} = a_{1,3}^{-1}a_{1,3} = a'_{3,1}(a'_{3,1})^{-1} = (a'_{3,1})^{-1}a'_{3,1} = 1, \\
&(2) \,\,\, a'_{2,4} = a_{2,4} - a_{1,3}^{-1}a_{1,4}a_{2,3}, \,\, a'_{2,1} = - a_{1,3}^{-1}a_{2,3}, \,\,
a'_{3,4} = a_{1,3}^{-1}a_{1,4},\,\, a'_{3,1} = a_{1,3}^{-1}, \\
&(3) \,\,\, a_{1,3} = (a'_{3,1})^{-1}, \,\, a_{1,4} = (a'_{3,1})^{-1}a'_{3,4}, \,\,
a_{2,3} = - (a'_{3,1})^{-1}a'_{2,1}, \\
&a_{2,4} = a'_{2,4} - (a'_{3,1})^{-1}a'_{3,4}a'_{2,1},
\end{split}
\]
with natural homomorphisms $\phi_{1,2;1}: R_1 \to R_{1,2}$ and $\phi_{1,2;2}: R_2 \to R_{1,2}$.

The algebras $R_{\lambda, \lambda'}$ for $\lambda, \lambda' \in \Lambda^{\max}$ 
such that $\vert \lambda \cap \lambda' \vert = 1$ are defined similarly
with natural homomorphisms
$\phi_{\lambda, \lambda';\lambda}: R_{\lambda} \to R_{\lambda, \lambda'}$ and 
$\phi_{\lambda, \lambda';\lambda'}: R_{\lambda'} \to R_{\lambda, \lambda'}$. 
\end{Defn}

\begin{Prop}
\[
\begin{split}
&R_{1,2} = k\langle a_{i,j}, a_{1,3}^{-1} \vert i = 1,2,\, j = 3,4 \rangle/(I, \, a_{1,3}a_{1,3}^{-1} - 1,\, a_{1,3}^{-1}a_{1,3} - 1) \\
&= k\langle a'_{i,j}, (a'_{3,1})^{-1} \vert i' = 2,3, \, j' = 4,1 \rangle/(I', \, a'_{3,1}(a'_{3,1})^{-1} - 1,\, (a'_{3,1})^{-1}a'_{3,1} - 1).
\end{split}
\]
\end{Prop}

\begin{proof}
For the first equality, it is sufficient to prove that, when we define the $a'_{i,j}$ by the relations (2), 
then the relations (3) and $I'$ follow.

(3) is a consequence of the equality $a'_{3,4}(a'_{3,1})^{-1} = (a'_{3,1})^{-1}a'_{3,4}$, which follows from $I'$.
We check that $I'$ follows from $I$:
\[
\begin{split}
&a'_{2,4}a'_{2,1} - a'_{2,1}a'_{2,4} \\
&= - (a_{2,4} - a_{1,3}^{-1}a_{1,4}a_{2,3})a_{1,3}^{-1}a_{2,3} + a_{1,3}^{-1}a_{2,3} (a_{2,4} - a_{1,3}^{-1}a_{1,4}a_{2,3}) \\
&= a_{1,3}^{-1}(- a_{1,3}a_{2,4} + a_{1,4}a_{2,3} + a_{2,4}a_{1,3} - a_{2,3}a_{1,4}) a_{1,3}^{-1}a_{2,3} = 0, \\
&a'_{3,4}a'_{3,1} - a'_{3,1}a'_{3,4} \\
&= a_{1,3}^{-1}a_{1,4}a_{1,3}^{-1} - a_{1,3}^{-1}a_{1,4}a_{1,3}^{-1}a_{1,4} = 0, \\
&a'_{2,4}a'_{3,1} - a'_{3,1}a'_{2,4} - a'_{2,1}a'_{3,4} + a'_{3,4}a'_{2,1} \\
&= (a_{2,4} - a_{1,3}^{-1}a_{1,4}a_{2,3})a_{1,3}^{-1} - a_{1,3}^{-1}(a_{2,4} - a_{1,3}^{-1}a_{1,4}a_{2,3}) \\
&+ a_{1,3}^{-1}a_{2,3}a_{1,3}^{-1}a_{1,4} - a_{1,3}^{-1}a_{1,4}a_{1,3}^{-1}a_{2,3} \\
&= a_{1,3}^{-1}(a_{1,3}a_{2,4} - a_{1,4}a_{2,3} - a_{2,4}a_{1,3} + a_{2,3}a_{1,4})a_{1,3}^{-1} = 0.
\end{split}
\]
The second equality is proved similarly.
\end{proof}

\begin{Cor}\label{Cor1}
(1) The set of closed points of $R_{1,2}$, i.e., the set of $k$-algebra homomorphisms $R_{1,2} \to k$, 
is an open subset of an affine space $\mathbf A^4$ with coordinates $\bar a_{i,j}$ for 
$i = 1,2$ and $j = 3,4$ defined by an equation $\bar a_{1,3} \ne 0$.

(2) The completion of $R_{1,2}$ at a closed point defined by a two-sided ideal $(a_{i,j} - a^0_{i,j})$ with $a^0_{i,j} \in k$
such that $a^0_{1,3} \ne 0$ is naturally isomorphic to the completion of $R_1$ at the same point.
\end{Cor}

\begin{proof}
(1) The abelianization of $R_{1,2}$ is a localization of a polynomial ring:
\[
R_{1,2}^{\text{ab}} = k[\bar a_{i,j}, \bar a_{1,3}^{-1} \vert i = 1,2,\,\,\, j = 3,4]
= k[\bar a'_{i',j'}, (\bar a'_{3,1})^{-1} \vert i' = 2,3,\,\,\, j' = 4,1]
\]
where the $\bar a_{i,j}$ and $\bar a'_{i',j'}$ are the images of the $a_{i,j}$ and $a'_{i',j'}$.

(2) This is a because $a_{1,3}$ is invertible at these points. 
\end{proof}

In order to define the change of variable formula from $R_1$ to $R_3$, we calculate the composition of the formulas
from $R_1$ to $R_2$ and from $R_2$ to $R_3$:

\begin{Lem}\label{R123}
Let 
\[
\begin{split}
&R_{1,2,3} = k\langle a_{i,j}, a'_{i',j'}, a''_{i'',j''}, a_{1,3}^{-1}, (a'_{2,4})^{-1} \vert i,j'' = 1,2,\, j,i = 3,4, \,\\
&i' = 2,3, \, j' = 4,1 \rangle/I_{1,2,3}
\end{split}
\]
where $I_{1,2,3}$ is a two-sided ideal generated by $I$ and equalities
$a_{1,3}a_{1,3}^{-1} = a_{1,3}^{-1}a_{1,3} = a'_{2,4}(a'_{2,4})^{-1} = (a'_{2,4})^{-1}a_{2,4} = 1$
as well as the relations
\[
\begin{split}
&a'_{2,4} = a_{2,4} - a_{1,3}^{-1}a_{1,4}a_{2,3}, \,\, a'_{2,1} = - a_{1,3}^{-1}a_{2,3}, \,\, a'_{3,4} = a_{1,3}^{-1}a_{1,4},\,\, 
a'_{3,1} = a_{1,3}^{-1}, \\
&a''_{3,1} = a'_{3,1} - (a'_{2,4})^{-1}a'_{2,1}a'_{3,4}, \,\, a''_{3,2} = - (a'_{2,4})^{-1}a'_{3,4}, \,\,a''_{4,1} = (a'_{2,4})^{-1}a'_{2,1}, \\
&a''_{4,2} = (a'_{2,4})^{-1}.
\end{split}
\]
Let $d = a_{1,3}a_{2,4} - a_{1,4}a_{2,3}$ and $d'' = a''_{3,1}a''_{4,2} - a''_{3,2}a''_{4,1}$.
Then the following hold.

(1) $dd'' = d''d = 1$.

(2) 
\[
\begin{split}
&a''_{3,1} = d^{-1}a_{2,4}, \,\,\, a''_{3,2} = - d^{-1}a_{1,4}, \,\,\, a''_{4,1} = - d^{-1}a_{2,3},\,\,\, a''_{4,2} = d^{-1}a_{1,3}, \\
&a_{1,3} = (d'')^{-1}a''_{4,2}, \,\,\, a_{1,4} = - (d'')^{-1}a''_{3,2}, \,\,\, a_{2,3} = - (d'')^{-1}a''_{4,1},\,\,\, a_{2,4} = (d'')^{-1}a''_{3,1}.
\end{split}
\]
\end{Lem}

\begin{proof}
(1) 
\[
\begin{split}
&dd'' = ((a'_{3,1})^{-1}(a'_{2,4} - (a'_{3,1})^{-1}a'_{3,4}a'_{2,1}) + (a'_{3,1})^{-1}a'_{3,4}(a'_{3,1})^{-1}a'_{2,1}) \\
&((a'_{3,1} - (a'_{2,4})^{-1}a'_{2,1}a'_{3,4})(a'_{2,4})^{-1} + (a'_{2,4})^{-1}a'_{3,4}(a'_{2,4})^{-1}a'_{2,1}), \\
&= (a'_{3,1})^{-1}(a'_{2,4}a'_{3,1} - a'_{2,1}a'_{3,4} + a'_{3,4}a'_{2,1})(a'_{2,4})^{-1} \\
&= (a'_{3,1})^{-1}(a'_{3,1}a'_{2,4} - a'_{3,4}a'_{2,1} + a'_{3,4}a'_{2,1})(a'_{2,4})^{-1} = 1.
\end{split}
\]
$d''d = 1$ is similarly proved.

(2)
\[
\begin{split}
&a''_{3,1} = a_{1,3}^{-1} + (a_{2,4} - a_{1,3}^{-1}a_{1,4}a_{2,3})^{-1}a_{1,3}^{-1}a_{2,3}a_{1,3}^{-1}a_{1,4} \\
&= (1 + (a_{1,3}a_{2,4} - a_{1,4}a_{2,3})^{-1}a_{2,3}a_{1,4})a_{1,3}^{-1} \\
&= (a_{1,3}a_{2,4} - a_{1,4}a_{2,3})^{-1}(a_{1,3}a_{2,4} - a_{1,4}a_{2,3} + a_{2,3}a_{1,4})a_{1,3}^{-1} \\
&= (a_{1,3}a_{2,4} - a_{1,4}a_{2,3})^{-1}a_{2,4}, \\
&a''_{3,2} = - (a_{2,4} - a_{1,3}^{-1}a_{1,4}a_{2,3})^{-1}a_{1,3}^{-1}a_{1,4} \\
&= - (a_{1,3}a_{2,4} - a_{1,4}a_{2,3})^{-1}a_{1,4}, \\
&a''_{4,1} = - (a_{2,4} - a_{1,3}^{-1}a_{1,4}a_{2,3})^{-1}a_{1,3}^{-1}a_{2,3} \\
&= (a_{1,3}a_{2,4} - a_{1,4}a_{2,3})^{-1}a_{2,3}, \\
&a''_{4,2} = (a_{1,3}a_{2,4} - a_{1,4}a_{2,3})^{-1}a_{1,3}.
\end{split}
\]
\end{proof}

\begin{Defn}
Let $d = a_{1,3}a_{2,4} - a_{1,4}a_{2,3} \in R_1$ and $d'' = a''_{3,1}a''_{4,2} - a''_{3,2}a''_{4,1} \in R_3$.
Then define 
\[
R_{1,3} = k\langle a_{i,j}, a''_{j,i}, d^{-1}, (d'')^{-1} \vert i = 1,2,\,\,\, j = 3,4 \rangle/I_{1,3}
\]
where $I_{1,3}$ is a two-sided ideal generated by the following relations:
\[
\begin{split}
&(1) \,\,\, I, \,\, I'', \,\, dd^{-1} = d^{-1}d = d''(d'')^{-1} = (d'')^{-1}d'' = 1, \\
&(2) \,\,\, a''_{3,1} = d^{-1}a_{2,4}, \,\, a''_{3,2} = - d^{-1}a_{1,4}, \,\, a''_{4,1} = - d^{-1}a_{2,3},\,\, a''_{4,2} = d^{-1}a_{1,3}, \\
&(3) \,\,\, a_{1,3} = (d'')^{-1}a''_{4,2}, \,\, a_{1,4} = - (d'')^{-1}a''_{3,2}, \,\, a_{2,3} = - (d'')^{-1}a''_{4,1},\,\,\\
&a_{2,4} = (d'')^{-1}a''_{3,1}, 
\end{split}
\]
with natural homomorphisms $\phi_{1,3;1}: R_1 \to R_{1,3}$ and $\phi_{1,3;3}: R_3 \to R_{1,3}$.

The algebras $R_{\lambda, \lambda'}$ for $\lambda, \lambda' \in \Lambda^{\max}$ 
such that $\vert \lambda \cap \lambda' \vert = 0$ are defined similarly
with natural homomorphisms
$\phi_{\lambda, \lambda';\lambda}: R_{\lambda} \to R_{\lambda, \lambda'}$ and 
$\phi_{\lambda, \lambda';\lambda'}: R_{\lambda'} \to R_{\lambda, \lambda'}$. 
\end{Defn}

\begin{Cor}
(1) The set of closed points of $R_{1,3}$
is an open subset of an affine space $\mathbf A^4$ with coordinates $\bar a_{i,j}$ for 
$i = 1,2$ and $j = 3,4$ defined by an equation $\bar a_{1,3}\bar a_{2,4} - \bar a_{1,4}\bar a_{2,3} \ne 0$.

(2) The completion of $R_{1,3}$ at a closed point defined by a two-sided ideal $(a_{i,j} - a^0_{i,j})$ with $a^0_{i,j} \in k$
such that $a^0_{1,3}a^0_{2,4} - a^0_{1,4}a^0_{2,3} \ne 0$ is naturally isomorphic to the completion of 
$R_1$ at the same point.
\end{Cor}

\begin{proof}
(1) The abelianization of $R_{1,3}$ is a localization of a polynomial ring:
\[
R_{1,3}^{\text{ab}} = k[\bar a_{i,j}, \bar d^{-1} \vert i = 1,2,\,\,\, j = 3,4]
= k[\bar a''_{i'',j''}, (\bar d'')^{-1} \vert i' = 3,4,\,\,\, j' = 1,2]
\]
where $\bar a_{i,j}, \bar d, \bar a''_{i'',j''}, \bar d''$ are the images of the $a_{i,j}, d, a''_{i'',j''}, d''$ respectively.

(2) If $a^0_{1,3} \ne 0$, then by Lemma \ref{R123}, the completion of $R_{1,3}$ is naturally isomorphic 
to the completion of $R_{1,2,3}$ at the same point, hence the assertion.
If $a^0_{1,4} \ne 0$, we take $\{2,4\}$ for $\lambda_2$ instead of $\{2,3\}$.
Then the assertion follows.
In the cases where $a^0_{2,3} \ne 0$ or $a^0_{2,4} \ne 0$ are treated similarly.
\end{proof}

Now we consider the intersections of $3$ or more basic open subsets.
In the case of $3$ open subsets, we have two subcases by symmetry:

(1) $\lambda_1 = \{1,2\}$, $\lambda_2 = \{2,3\}$ and $\lambda_3 = \{3,4\}$.

(2) $\lambda_1 = \{1,2\}$, $\lambda_2 = \{2,3\}$ and $\lambda_4 = \{1,3\}$.
 
In the first case, we take $R_{1,2,3}$ as the corresponding algebra which was considered already.
The abelianization is
\[
R_{1,2,3}^{\text{ab}} = k[\bar a_{i,j}, \,\bar a_{1,3}^{-1},\, \bar d^{-1} \, \vert \,\, i = 1,2,\,\, j = 3,4].
\]

In the second case, the $3$ vertexes of the octahedron span a regular triangle.
We write $R_4 = R_{\lambda_4}$, $a'''_{i,j} = a_{i,j}^{\lambda_4}$ and $I''' = I_{\lambda_4}$.
Thus $I'''$ is generated by 
\[
\begin{split}
&a'''_{12}a'''_{1,4} - a'''_{1,4}a'''_{1,2} = 0, \,\,\, a'''_{3,2}a'''_{3,4} - a'''_{3,4}a'''_{3,2} = 0, \\
&a'''_{1,2}a'''_{3,4} - a'''_{1,4}a'''_{3,2} + a'''_{3,2}a'''_{1,4} - a'''_{3,4}a'''_{1,2} = 0.
\end{split}
\]
We define
\[
\begin{split}
&R_{1,2,4} = k\langle a_{i,j}, a'_{i',j'}, a'''_{i'',j''}, a_{1,3}^{-1}, a_{2,3}^{-1} \, \vert\,\, i = 1,2,\, j = 3,4, \, \\
&i' = 2,3, \, j' = 4,1, \, i''' = 1,3, \, j''' = 2,4 \rangle/I_{1,2,4}
\end{split}
\]
where $I_{1,2,4}$ is a two-sided ideal generated by $I$ and equalities
$a_{1,3}a_{1,3}^{-1} = a_{1,3}^{-1}a_{1,3} = a_{2,3}a_{2,3}^{-1} = a_{2,3}^{-1}a_{2,3} = 1$
as well as the relations
\[
\begin{split}
&a'_{2,4} = a_{2,4} - a_{1,3}^{-1}a_{1,4}a_{2,3}, \,\, a'_{2,1} = - a_{1,3}^{-1}a_{2,3}, \,\, a'_{3,4} = a_{1,3}^{-1}a_{1,4},\,\, 
a'_{3,1} = a_{1,3}^{-1}, \\
&a'''_{1,4} = a_{1,4} - a_{2,3}^{-1}a_{2,4}a_{1,3}, \,\, a'''_{1,2} = - a_{2,3}^{-1}a_{1,3}, \,\,a'''_{3,4} = a_{2,3}^{-1}a_{2,4}, 
\,\, a'''_{3,2} = a_{2,3}^{-1}.
\end{split}
\]
The abelianization is
\[
R_{1,2,4}^{\text{ab}} = k[\bar a_{i,j}, \,\bar a_{1,3}^{-1}, \,\bar a_{2,3}^{-1} \, \vert \,\, i = 1,2,\,\, j = 3,4].
\]
The rest of the proof is the same as before.

The case of the intersections of 4 or more basic open subsets are treated similarly and the calculations follows 
from the cases above. 

The universal bundles glue together naturally along the natural restriction homomorphisms.


Graduate School of Mathematical Sciences, University of Tokyo,
Komaba, Meguro, Tokyo, 153-8914, Japan. 

kawamata@ms.u-tokyo.ac.jp

\end{document}